\documentclass{article}
\usepackage{graphicx} 

\usepackage{amsmath,amsthm,amssymb,amsfonts}
\usepackage{graphicx} 
\usepackage{enumerate}
\usepackage{blkarray}
\usepackage{authblk}
\usepackage{subfigure}
\usepackage{array}
\usepackage{stackengine}
\usepackage{faktor}
\usepackage{fancyhdr}
\usepackage{graphicx}
\usepackage[svgnames]{xcolor}
\usepackage[labelfont=bf,width=0.8\textwidth,justification=centering]{caption}
\usepackage{tikz}
\usepackage{tikz-cd}
\usepackage{subfigure}
\usepackage[colorlinks=true,linkcolor=Indigo,citecolor=Indigo]{hyperref
 }
\usepackage{color}

\def\P{\bar{P}}
\newtheorem{theorem}{Theorem}[section]

\newtheorem{lemma}[theorem]{Lemma}
\theoremstyle{definition}
\newtheorem{observation}[theorem]{Observation}
\newtheorem{remark}[theorem]{Remark}
\newtheorem{example}[theorem]{Example}
\numberwithin{equation}{section}
\numberwithin{figure}{section}
  
  \title{Adaptive dynamics for individual payoff game-theoretic models of vaccination }
\author[1]{Nataliya Balabanova}
\author[2]{Manh Hong Duong}
\affil[1,2]{School of Mathematics, University of Birmingham, UK}
\date{}
\begin{document}

\maketitle

\begin{abstract}
Vaccination is widely recognised as one of the most effective forms of public health interventions. Individuals decisions regarding vaccination creates a complex social dilemma between individual and collective interests, where each person’s decision affects the overall public health outcome. In this paper, we study the adaptive dynamics for the evolutionary dynamics of strategies in a fundamental game-theoretic model of vaticination. We show the existence of an (Nash) equilibrium and analyse the stability and bifurcations when varying the relevant parameters. We also demonstrate our analytical results by several concrete examples.
\end{abstract}
\section{Introduction}
\label{sec: introduction}
According to the World Health Organisation (WHO), immunization is one of the most effective ways of public health interventions and vaccination plays a crucial role in controlling the spread of infectious diseases; this is backed by a wealth of research data (\cite{WHO, 10.1001/archpedi.159.12.1136, amanna2008protective, hamborsky2015epidemiology, kim2016advisory}). Individual decisions regarding vaccination are often influenced by personal perceptions of risk and reward.  In the context of vaccination, individuals weigh the benefits of immunization—such as personal protection and contributing to herd immunity—against the perceived risks, such as side effects or the belief that they may not need to vaccinate if others do. This creates a complex social dilemma between individual and collective interests, where each person’s decision affects the overall public health outcome ~\cite{bauch2003group,bauch2004vaccination}. 

Because of these reasons, game theory, which is a well-established  mathematical framework used to analyze strategic decision-making, offers valuable insights into these vaccination behaviors. There is now a large body of literature on using game theory to study vaticination. Using game theoretic approaches allows researchers to predict how people will behave in different vaccination scenarios, such as when disease prevalence is low or when vaccine scares occur. In addition, game-theoretical approaches also reveal that individuals may choose not to vaccinate, relying on others to maintain herd immunity, a behavior that can lead to suboptimal vaccination coverage and the resurgence of preventable diseases.  For instance, the public opposition movement in the USA  has become significantly more well-organised and networked in the past two decades (\cite{carpiano2023confronting}); while parental vaccine hesitation worldwide has been also listed as a growing concern (\cite{mckee2016exploring}). This analysis highlights the importance of understanding strategic interactions among individuals in the design of effective vaccination policies. It also provides a foundation for exploring how incentives, public information campaigns, and policy interventions can be optimized to encourage higher vaccination uptake and ensure widespread immunity across populations. We refer the reader to ~\cite{bauch2003group,bauch2004vaccination,chapman2012using,hausken2022game,traulsen2023individual} and references therein for more information on this topics.

In this paper, we adopt the adaptive dynamics to analyse the evolutionary dynamics of strategies in game models of vaticination \cite{bauch2004vaccination}. By its definition, the adaptive dynamics predicts the most profitable strategy for a deviating minority to adopt under given conditions. Therefore, it is ideally suited in models that are honed to predict the behaviour of the individuals acting in what they believe is their best interest. The adaptive dynamics has been used extensively in a huge body of works in population dynamics and evolutionary game theory and other fields, including researches on infectious diseases, see for instance~\cite{van2002adaptive}. We also refer the reader to the monograph \cite{hofbauer1998evolutionary} and recent surveys \cite{brannstrom2013hitchhiker,de2024dynamics} for a detailed account of this topics including many applications.  

We describe our model and specify the assumptions in detailed in Section \ref{sec: adaptive dynamics}.  In our setting, the adaptive dynamics is given by a first order differential equation with two parameters $r$ and $\epsilon$ -- respectively the relative risk of vaccination and the percentage of people deviating from the population's strategy. We analyse the model from the point of view of differential equations and bifurcation theory (see \cite{crawford1991introduction, kuznetsov1998elements}). In Section \ref{sec: stability analysis}, we show that a (Nash) equilibrium always exists; in fact, there can be more than one equilibrium state, alternating between stable and unstable. In Section \ref{sec: parameters and bifurcations} we investigate the effect that the parameters have on the number and stability of equilibrium states. Lastly,  in Section  \ref{sec: examples} we consider two concrete examples of interest: one with a singular equilibrium and the other with multiple equilibria, and draw bifurcation diagrams for the latter.

\section{The game-theoretic model for vaccination and the adaptive dynamics}
\label{sec: adaptive dynamics}
\subsection{A game-theoretic model for vaccination}
We consider a game model of vaccination following the approach in the seminal paper \cite{bauch2004vaccination}.  In this model, all individuals in the population receive the same information and use it uniformly to assess risks. An individual’s strategy is defined by the probability, $P$, that they will choose to get vaccinated. The overall vaccine uptake in the population represents the proportion of newborns who are vaccinated, which reflects the average of the strategies adopted by individuals. We assume there is no delay between changes in vaccine uptake and subsequent changes in overall vaccine coverage. Therefore, if no disease- or vaccine-related deaths occur, the proportion of vaccinated individuals, $p$, will match the vaccine uptake rate. The payoff to an individual will be greater when the morbidity risk, which is the probability of adverse consequences, is lower. Let $\pi(p)$ denote the probability that an unvaccinated individual will eventually be infected if the vaccine coverage level in the population is $p$. With this notation, the strategy of vaccinating with probability $P$ yields an expected payoff given by \cite{bauch2004vaccination}
\begin{equation}
\label{eq:general payoff}
E(P,p)=-rp-\pi(p)(1-P),    
\end{equation}
where $0<r<1$ is a parameter (the so-called the relative risk), which the ratio of the risk of infection to a vaccinated individual to the risk of the  infection to an unvaccinated individual.

To proceed further, we suppose that a proportion $\epsilon$ ($0\leq \epsilon\leq 1$) of the population vaccinates with a probability $P$ and the remainder vaccinate with a probability $Q$. Then the common  strategy $p$ on the entire group can be expressed as 
 \begin{equation}
 \label{eq:population composition}
     p = \epsilon P + (1-\epsilon)Q.
      \end{equation}
Substituting this expression into \eqref{eq:general payoff}, the (normalised) individual payoff of a $P$-strategist in the population of composition (\ref{eq:population composition}) is given by
 \begin{equation}
     \label{eq: payoff}
     E(P, \epsilon P + (1-\epsilon)Q) = -rP-\pi\left(\epsilon P + (1-\epsilon)Q\right)(1-P).
 \end{equation}
Using this payoff function, we will formulate the adaptive dynamics to study the evolutionary dynamics of the strategy $P$. To this end, we first recall the mathematical formulation of the adaptive dynamics in a general setting.
\subsection{Adaptive dynamics}
We consider a simplified model of population dynamics, that consists of an asexually reproducing resident population characterised by some trait (strategy) $x$. The population is regularly challenged by mutants with trait (strategy) $y$. The new mutant may successfully invade and take over the resident population or die out, depending on the fitness advantage (or disadvantage) that the mutant faces in the homogeneous population of residents. This is characterised by a payoff function~$A(y,x)$, depending on both $x$ and $y$. More specifically, mutants successfully invade and subsequently become the majority if and only if the following inequality holds \cite{hofbauer1998evolutionary}
\[
A(y,x)> A(x,x).
\]
This condition means that, mutants receive a higher payoff against residents than the residents receive themselves. Thus the invasion fitness quantifies the evolutionary advantage of mutations. One can also formally write the evolutionary equation for the direction of the mutation that yields the largest invasion fitness. Assuming that the  mutant trait~$y$ is infinitesimally close to the resident trait~$x$. At every subsequent time step, the mutant with the largest payoff advantage reaches fixation and forms the new resident population. This new population is then again challenged by those mutants that have the largest advantage in the new environment, and so on. This resulting evolutionary process can be mathematically characterized by the adaptive dynamics, first proposed by Hofbauer and Sigmund in 1990, which is given by (see e.g. \cite[Chapter 9]{hofbauer1998evolutionary}):
\begin{equation}
    \label{eq: general adaptive dynamics}
    \dot{y} = \frac{\partial A(x,y)}{\partial y}\vert_{y=x}.
\end{equation}
From a mathematical point of view, this is in general a nonlinear ordinary differential equation. The specific form of the payoff function depends on concrete applications. We refer the reader to the monograph \cite{hofbauer1998evolutionary} and recent surveys \cite{brannstrom2013hitchhiker,de2024dynamics} for a detailed account of the adaptive dynamics, which also include many interesting applications.
\subsection{Adaptive dynamics for the game-theoretic model of vaccination}
We now apply the adaptive dynamics \eqref{eq: general adaptive dynamics} to study the evolutionary dynamics of the strategy $P$ using the payoff function \eqref{eq: payoff}. We first compute the derivative of the payoff function with respect to $P$:
 \begin{equation}
 \begin{split}
     \frac{\partial E(P,\epsilon P + (1-\epsilon)Q) }{\partial P}\vert_{P=Q} &= \bigl(-r - \epsilon \pi'(\epsilon P + (1-\epsilon) Q)(1-P)  \\&+ \pi(\epsilon P + (1-\epsilon) Q)\bigr)\vert_{P=Q}\\&= -r-\epsilon\pi'(P)(1-P)+  \pi(P).
 \end{split}
 \end{equation}
 Therefore, the adaptive dynamics for the dynamics of $P$ is given by
 \begin{equation}
\label{eq: adaptive dynamics}
\dot{P}=\frac{\partial E(P,\epsilon P + (1-\epsilon)Q) }{\partial P}\vert_{P=Q} = -r - \epsilon\pi'(P)(1-P)  + \pi(P).
 \end{equation}
From a mathematical point of view, the equation (\ref{eq: adaptive dynamics}) is an nonlinear ordinary differential equation (ODE) in one variable, with two parameters $r$ and $\epsilon$. In the next section, we will analyse this equation from the point of view of differential equations and bifurcation theory. In particular, we will determine the possible number of equilibria and their stability and describe the bifurcations arising from varying the parameters $r$ and $\epsilon$. We will study these properties with a general payoff function $\pi(p)$ in the next two sections, then we provide concrete examples in the last section.

\section{Solutions and equilibrium analysis}
\label{sec: stability analysis}
\subsection*{Assumptions on the payoff function $\pi(p)$}
We assume that in our model  getting vaccinated completely eliminates the possibility of becoming infected and that the increase in numbers of vaccinated people leads to a decrease in new cases and thus lowers the probability of virus exposure for the focal individual. After some threshold $p^{\ast}$ of rate of vaccination is surpassed, there is practically no chance of encountering an infected person; thus, $\pi(p)$ tuns zero. Note that the value of $p^{\ast}$ in our model 
is arbitrary and can be raised to $1$.

These natural assumptions  entail that the function $\pi(p)$ must have the following properties:
\begin{enumerate}
    \item $\pi$ is smooth; this property is usually a matter of convenience, but in this case it plays a crucial role, as we will see later;
    \item $0<\pi(p)<1$ when $p\in[0,1]$;
    \item $\pi$ is monotonously decreasing on $[0,1]$;
    \item $\pi(0)=1$;
    \item $\pi(p) = 0$ when $p$ is greater than some fixed $p^{\ast}$
\end{enumerate}
\begin{figure}
\centering
\includegraphics[scale=.7]{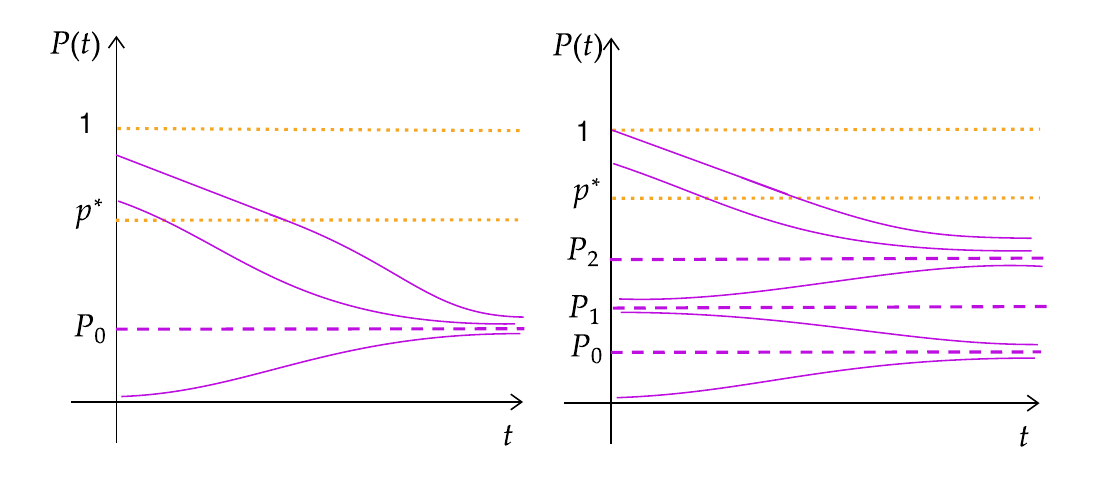}
\caption{Schematic phase portrait of $P(t)$, depending on the number of equilibria of (\ref{eq: equilibria}). Equilibria are denoted by $P_i$ and dashed purple lines, non-constant solutions-by non-dashed purple lines and intrinsic constants 1 and $p^{\ast}$ -- by orange dotted lines.}
\label{fig: phase portrait}
\end{figure}

We assume these properties throughout the paper. In this section, we start with investigating the solutions of (\ref{eq: adaptive dynamics}), focusing on existence and stability of equilibria. The parameters $r$ and $\epsilon$ are assumed to be fixed throughout this section.

Our first theorem below is on the existence of equilibria of the adaptive dynamics \eqref{eq: adaptive dynamics}.
 \begin{theorem}
\label{st: existence of zeros}
\begin{enumerate}[(i)]
    \item The equation (\ref{eq: adaptive dynamics}) has at least one equilibrium point.  
\item If $\pi$ is convex then it has a unique stable (Nash) equilibrium, which is decreasing as a function of the relative risk $r$, but increasing as a function of the proportion $\epsilon$. In addition, it belongs to the interval
\[
\left[\pi^{-1}(r), \min\left\{p^{\ast},\pi'^{-1}\left(\frac{1-r}{\epsilon}\right)\right\}\right].
\]
\end{enumerate}
\end{theorem}
\begin{figure}
    \centering
   \subfigure[One intersection
]{\includegraphics[scale = .393]{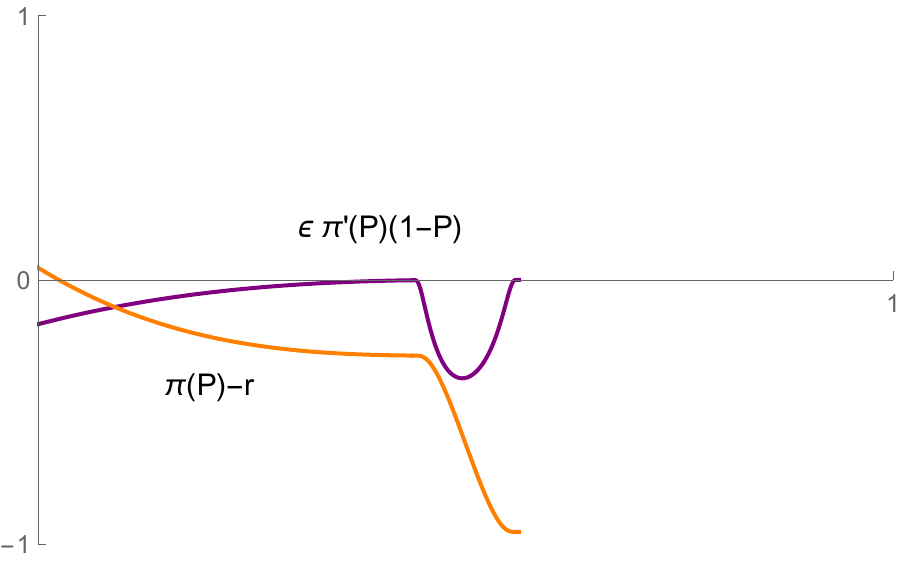}}
      \subfigure[Multiple intersections]{\includegraphics[scale = .393]{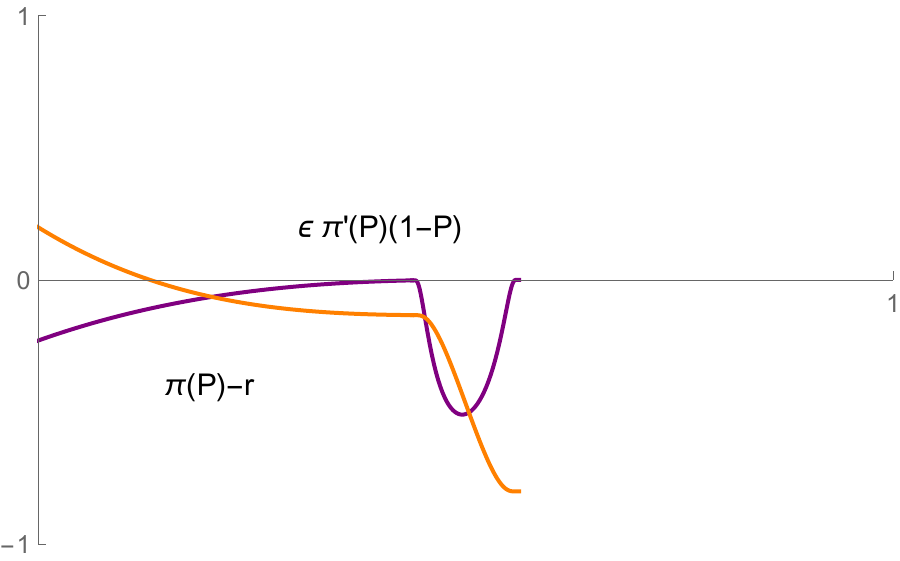}}
      \caption{Varying number of intersections depending on the form f $\pi(p)$}
    \label{fig:intersections}
\end{figure}
\begin{proof}
(i) The condition for equilibria of (\ref{eq: adaptive dynamics}) is 
    \begin{equation*}
         f(P):=-r - \epsilon\pi'(P)(1-P)  + \pi(P)=0
    \end{equation*}
The function $f: [0,1]\rightarrow \mathbb{R}$ is continuous. Using the assumptions on $\pi$, we have
\[
f(0)=-r-\epsilon\pi'(0)+\pi(0)= (1-r)-\epsilon \pi'(0)>0,
\]
since $0<r<1$ and $\pi'(0)\leq 0$, which is because $\pi$ is monotonously decreasing on $[0,1]$.  We also have
\[
f(1)=-r+\pi(1)=-r<0.
\]
It follows that, since $f$ is continuous, it has at least one root in $[0,1]$. 

We can also demonstrate the existence of an equilibrium geometrically. We can rewrite the equation $f(P)=0$  as 
    \begin{equation}
        \label{eq: equilibria}
     \pi(P)  -r = \epsilon\pi'(P)(1-P)  
    \end{equation}
    The left hand side of (\ref{eq: equilibria}) is a decreasing function with a maximum of 
    $1-r$ at 0 and a value of $-r$ when $P>p^{\ast}$. On the other hand, the function $\epsilon \pi'(P)(1-P)$ is a negative function that turns 0 when $P>p^{\ast}$. Therefore, the two functions must intersect at least in one point with the intersection lying between $0$ and $p^{\ast}$. See also Figure \ref{fig:intersections} for an illustration of possible scenarios.

(ii) Now we suppose that $\pi$ is convex. We will show that there is a unique Nash equilibrium and study its behaviours with respect to the parameters $r$ and $\epsilon$. The derivative of $f$ with respect to $P$ is given by
\[
f'(P)=\pi'(P)(1+\epsilon)-\epsilon(1-P)\pi''(P)\leq 0,
\]
since $\pi'(P)\leq 0$ and $\pi''(P)\geq 0$. It follows $f$ is monotonously decreasing on $[0,1]$, therefore it has a unique root $P_0$, which is stable. Therefore it is a Nash equilibrium. To study the behaviour of $P_0$ with respect to the relative risk $r$, we consider $P_0$ as a function of $r$. By the implicit theorem, we have
\[
\frac{\partial P_0}{\partial r}=-\frac{\partial_r f}{\partial_P f}(r,P_0)=\frac{1}{\pi'(P_0)(1+\epsilon)-\epsilon(1-P_0)\pi''(P_0)}\leq 0.
\]
Thus $P_0$ is decreasing as a function of the relative risk $r$. Similarly, considering $P_0$ as a function of $\epsilon$, we get
\[
\frac{\partial P_0}{\partial \epsilon}=-\frac{\partial_\epsilon f}{\partial_P f}(\epsilon,P_0)=\frac{\pi'(P_0)(1-P_0)}{\pi'(P_0)(1+\epsilon)-\epsilon(1-P_0)\pi''(P_0)}\geq 0.
\]
Thus $P_0$ is increasing as a function of $\epsilon$. It remains to establish the lower bound and upper bound for $P_0$. 

The function $f(P)$ is monotonously decreasing; additionally, since $\pi(P)$ is a monotonous function, there exists a unique preimage of $r$. The value $f(\pi^{-1}(r))$ will therefore be
   \[
   r - r - \epsilon \pi'(\pi^{-1}(r))(1 - \pi^{-1}(r))=- \epsilon \pi'(\pi^{-1}(r))(1 - \pi^{-1}(r))>0,
   \]
forcing $P_0>\pi^{-1}(r)$.

For the upper bound, we note that 
\[
f(P)<1-\frac{P}{p^{\ast}}-r -\epsilon\pi'(P)(1-P)=:h(P),
\]
since $\pi(P)\le1 - P/p^{\ast}$ due to concavity of $\pi$. We have
\[
h'(P)=-\frac{1}{p^*}-\epsilon \pi''(P)(1-P)+\epsilon \pi'(P)<0
\]
since $0<p^*<1, \pi''(P)\geq 0, \pi'(P)\leq 0$. Therefore $h$ is monotonously decreasing. In addition, since $\pi'(0)\leq 0$, $\pi'(1)=0$, and $0<p^*<1$ we have
\[
h(0)=1-r-\epsilon \pi'(0)>0 \quad\text{and}\quad h(1)=1-\frac{1}{p^*}-r<0.
\]
It implies that $h$ has a unique solution $\P$ with $0<\P<1$, that is $0<\P<1$ satisfies the following equation:
   \begin{equation}
   \label{eqn}
   1 - \frac{\P}{p^{\ast}}-r - \epsilon \pi'(\P)(1-\P) = 0.   
   \end{equation}
Since $f(P)<h(P)$ and both functions are monotonously decreasing, for any $\alpha>0$ we have
\[
0=f(P_0)= h(\P)> h(\P+\alpha)> f(\P+\alpha)
\]
Therefore, $P_0<\P+\alpha$ for all $\alpha>0$, which implies that $P_0\leq \P$.

Equation \eqref{eqn} can be transformed as follows: 
   \begin{equation*}
       \begin{split}
\epsilon \pi'(\P)(1-\P) &= \frac{p^{\ast}(1-r) - \P}{p^{\ast}} \\
\pi'(\P) &=\frac{1}{\epsilon p^{\ast}}\frac{p^{\ast}(1-r) - \P}{1-\P}\\
 \pi'(\P) &=\frac{1}{\epsilon p^{\ast}}\frac{\P-p^{\ast}(1-r) }{\P-1} \\
 \pi'(\P) &=\frac{1}{\epsilon p^{\ast}}\left(1 + \frac{1 - p^{\ast}(1-r)}{\P-1}\right).
       \end{split}
   \end{equation*}
Let
\[
g(P):= \frac{1}{\epsilon p^{\ast}}\left(1 + \frac{1 - p^{\ast}(1-r)}{P-1}\right).
\]
Then, since $1 - p^{\ast}(1-r)>0$, we have
\[
g'(P)=-\frac{1}{\epsilon p^*}\frac{1-p^*(1-r)}{(P-1)^2}<0.
\]
Therefore $g$ is monotonously decreasing. It implies that, since $\P>0$,
\begin{equation}
\label{eq: inequality}
\pi'(\P)=g(\P)\leq g(0)=\frac{1}{\epsilon p^{\ast}}\left(1 + \frac{1 - p^{\ast}(1-r)}{-1}\right)=\frac{1-r}{\epsilon}.
\end{equation}
Since $\pi(P)$ is convex, $\pi''(P)$ is a positive function. Thus $\pi'(P)$ is an increasing function (albeit a negative one). Hence from the inequality \eqref{eq: inequality} we deduce that
\[
\P\leq \pi'^{-1}\left[\frac{1-r}{\epsilon}\right].
\]
This finishes the proof of this theorem.

\end{proof}
   \begin{remark}
We note that no general claims (independent of $r$ and $\epsilon$) can  be made about the upper or lower  bounds of the equilibrium $P_0$, since as $r\to 0$, $P_0\to p^{\ast}$ and $P_0\to0 $ as $r\to1$. In addition, in general since $f(P)$ is a nonlinear function, there is no explicit formula for $P_0$. One can use various numerical methods (such as  Newton's method) to compute $P_0$. We provide concrete examples of the payoff function and numerically compute the equilibria in Section \ref{sec: examples}.  
    \end{remark}
    We call the equilibria, in order of increase in $P$-coordinate, $P_0$, $P_1$ and so on, with the general notation $P^{\ast}$ for the  set of all equilibria.
        Figure \ref{fig: phase portrait} illustrates the phase portrait of $P(t)$ depending on the number of rest states.  

    \begin{remark}
        Observe that when $P>p^{\ast}$, the derivative $\dot{P}$ is constant and equal to $-r$; hence the function $P(t)$ is linear when $P>p^{\ast}$. 
    \end{remark}
The following lemma establishes the stability analysis for the equilibria.
\begin{lemma}
\label{st: number of equilibria}
The following hold for the equation (\ref{eq: equilibria}):
\begin{enumerate}
    \item There can only be an odd number of non-degenerate equilibrium points.
    \item If all equilibria are nondegenerate, they alternate between stable and unstable, starting with stable. 
\end{enumerate}
\end{lemma}
\begin{proof}

    The first statement immediately follows from the proof of Theorem \ref{st: existence of zeros}, namely, the monotonicity of the left hand side of (\ref{eq: equilibria}). The graph of the function $\epsilon\pi'(P)(1-P)$ has to end up above that of $\pi(P)-r$, yet it starts below; therefore, there can only be an odd number of transversal intersections. 

    The second statement is evident from the geometric from of the intersections - see Figure \ref{fig:intersections} (a) and (b). An equilibrium is stable when $\dot{P}$ turns from positive to negative and unstable otherwise; the former happens at odd number intersections and the latter at ones that have even number.

\end{proof}
\begin{figure}
 \subfigure[Varying $r$ with fixed $\epsilon\ne 0$]{\includegraphics[scale=.375]{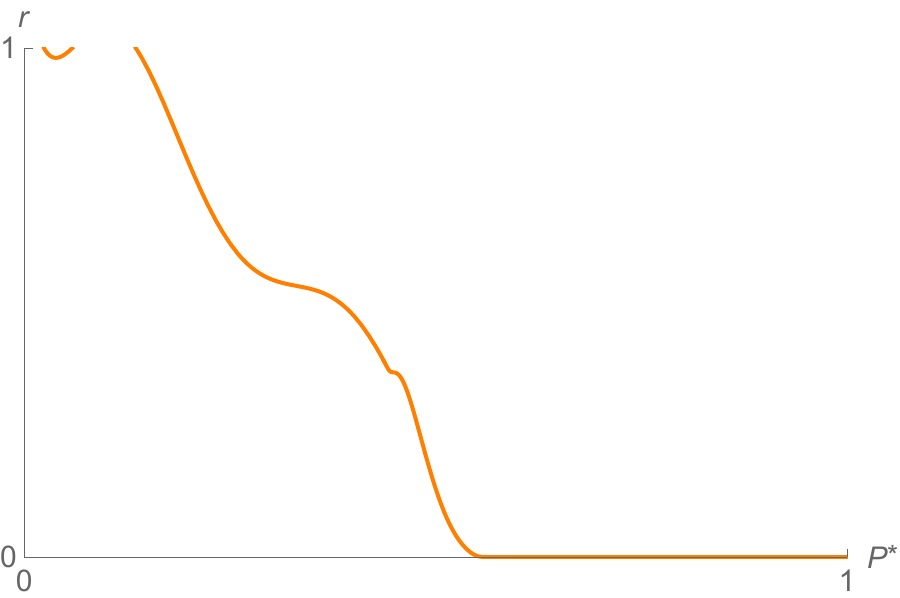}}
 \subfigure[Varying $\epsilon$ with fixed $r\ne 0$]{\includegraphics[scale=.375
]{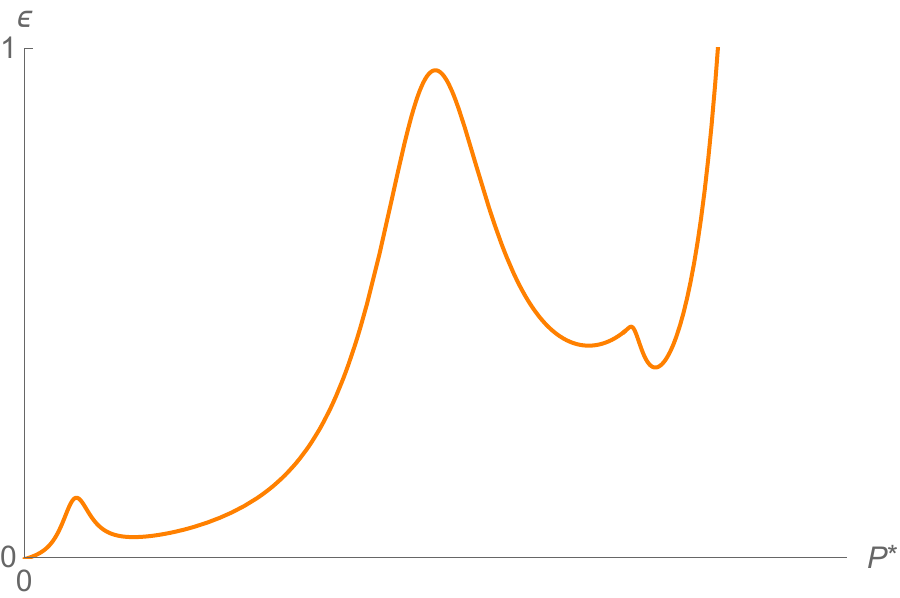}}
\caption{Possible bifurcation diagrams of equilibria depending on parameters}
\label{fig: bifurcation diagrams}
\end{figure}
When an equilibrium is unique, it is stable and therefore Nash; deviating from it would result in a loss to the individual. In the presence of multiple equilibria, only stable ones will be Nash, as the individual is incentivised to move away from the unstable ones.  
\section{Parameters and bifurcations}
\label{sec: parameters and bifurcations}
In this section, we analyse the bifurcations of the equation (\ref{eq: adaptive dynamics}) stemming from the presence of the two important parameters $\epsilon$ and $r$, where we recall from Section \ref{sec: adaptive dynamics} that $\epsilon$ is the proportion of the population vaccinated with a probability $P$ and $r$ is the relative-risk. We start with the simpler case, that of the relative-risk, $r$.

Note that in this section we will denote the axis on which the equilibria lie by $P^{\ast}$ -- not to be confused with the epidemiological parameter $p^{\ast}$.
\subsection{Parameter $r$}
\label{sec: parameter r}
In order to study the bifurcations, we fix some general (for our purposes that effectively means that it is nonzero) value of $\epsilon$ and observe how varying $r$ affects solutions of (\ref{eq: adaptive dynamics}). 

Varying $r$ slides the graph of $\pi(P)-r$ up and down; the equilibria slide with it, staying on the curve. 

\begin{observation}
\label{st:obs equilibria}
\begin{enumerate}
\item The only type of bifurcations that happens is \textit{saddle-node}, i.e. appearance of a degenerate equilibrium and its subsequent change into a stable-unstable-pair or a merge of a stable-unstable pair into a degenerate equilibrium and its disappearance. 
\item When a stable-unstable pair appears, the value of the stable equilibrium decreases with the increase of $r$, and the value of the unstable one increases. Thus, merging happens between a stable and an unstable equilibria that came from neighbouring degenerate ones. 
\end{enumerate}
\end{observation}

The equilibria corresponding to $r=0$ are all points with $P>p^{\ast}$; ones corresponding to $r=1$ are those of the equation $\pi(P)-1 = \epsilon \pi'(P)(1-P)$.

As for the shape of the bifurcation diagram, observe that we may return to the original form and treat our problem as locating zeros of $\pi(P)-\epsilon\pi'(P)(1-P)-r$; we observe how the zeros change when we move this graph up and down; this allows us to conclude that

\begin{lemma}
    
The bifurcation diagram for $r$ is the $\pi/2$-rotated graph of $\pi(P)-\epsilon\pi'(P)(1-P)$ that was cut off at   $\pi(P)-\epsilon\pi'(P)(1-P) = 0$ and  $\pi(P)-\epsilon\pi'(P)(1-P)=1$.  
\end{lemma}
    Figure \ref{fig: bifurcation diagrams} (a) shows a potential form of the bifurcation diagram; the set of equilibria is on the horizontal axis and denoted by $P^{\ast}$. 
    \begin{figure}
        \centering
     \includegraphics[scale=.4]{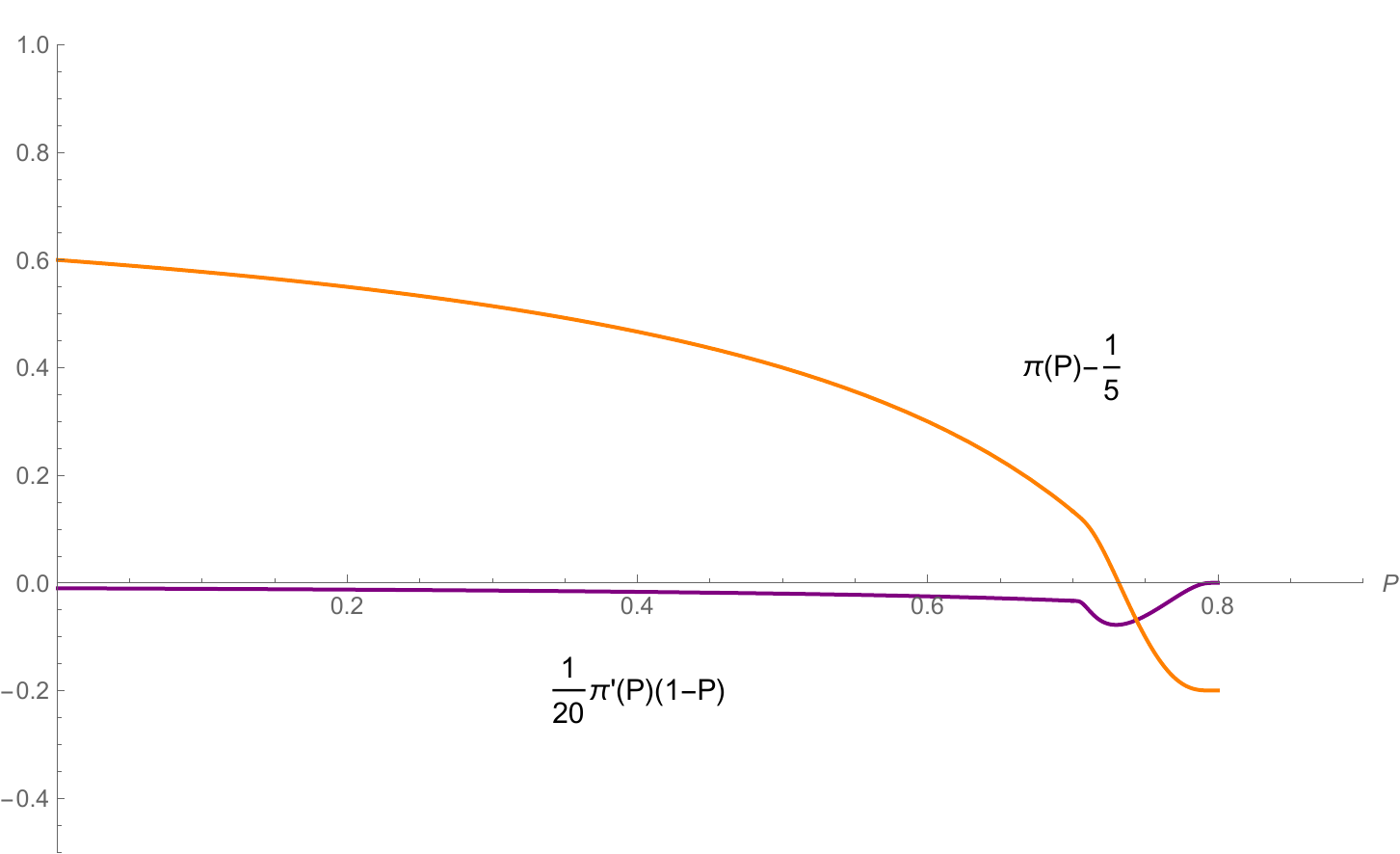}
        \caption{Intersections in Example \ref{ex: concrete stuff}}
        \label{fig: intersections example}
    \end{figure}

\subsection{Parameter $\epsilon$}
Now we move on to the second parameter, $\epsilon$. The bifurcations in it are slightly more complex than those in $r$, but we can make the same remarks as those we did about bifurcations in $r$ in Observation \ref{st:obs equilibria}. 

Note that $\pi'(P)(1-P)$ does not assume zero values unless $P>p^{\ast}$; unless $r=0$, the equation (\ref{eq: adaptive dynamics}) does not have zeros in the interval $[p^{\ast},1]$. 

So if we assume that $r>0$ and $P<p^{\ast}$, we can rewrite (\ref{eq: equilibria}) as 
\begin{equation}
\label{eq: varying epsilon}
\frac{\pi(P)-r}{\pi'(P)(1-P)} -\epsilon= 0.
\end{equation}

Thus, we can state that 
\begin{observation}
When $r\ne 0$, the bifurcation diagram in $\epsilon $ is the $\pi/2$ rotated graph of the function $\frac{\pi(P)-r}{\pi'(P)(1-P)}$, cut off at the lines $\frac{\pi(P)-r}{\pi'(P)(1-P)} = 0$ and $\frac{\pi(P)-r}{\pi'(P)(1-P)}=1$

\end{observation}

\begin{figure}
\subfigure[$\epsilon = 0.078$]{\includegraphics[scale=.35]{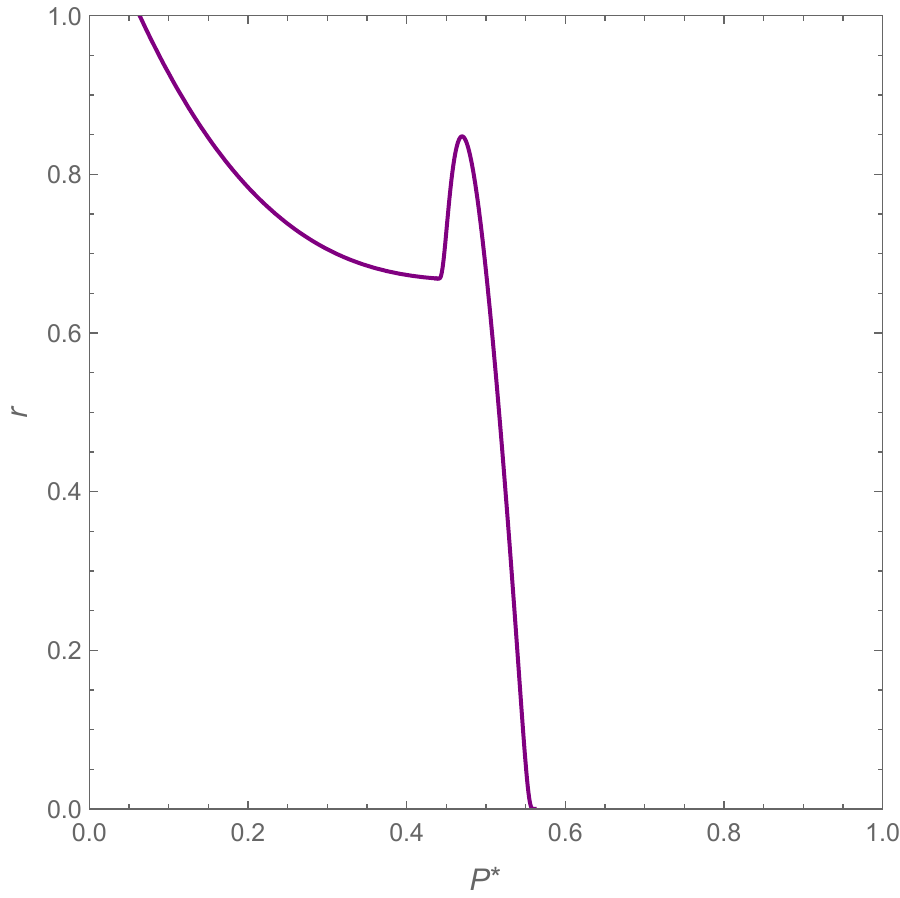}}
\subfigure[$\epsilon = 0.93
$]{\includegraphics[scale=.35]{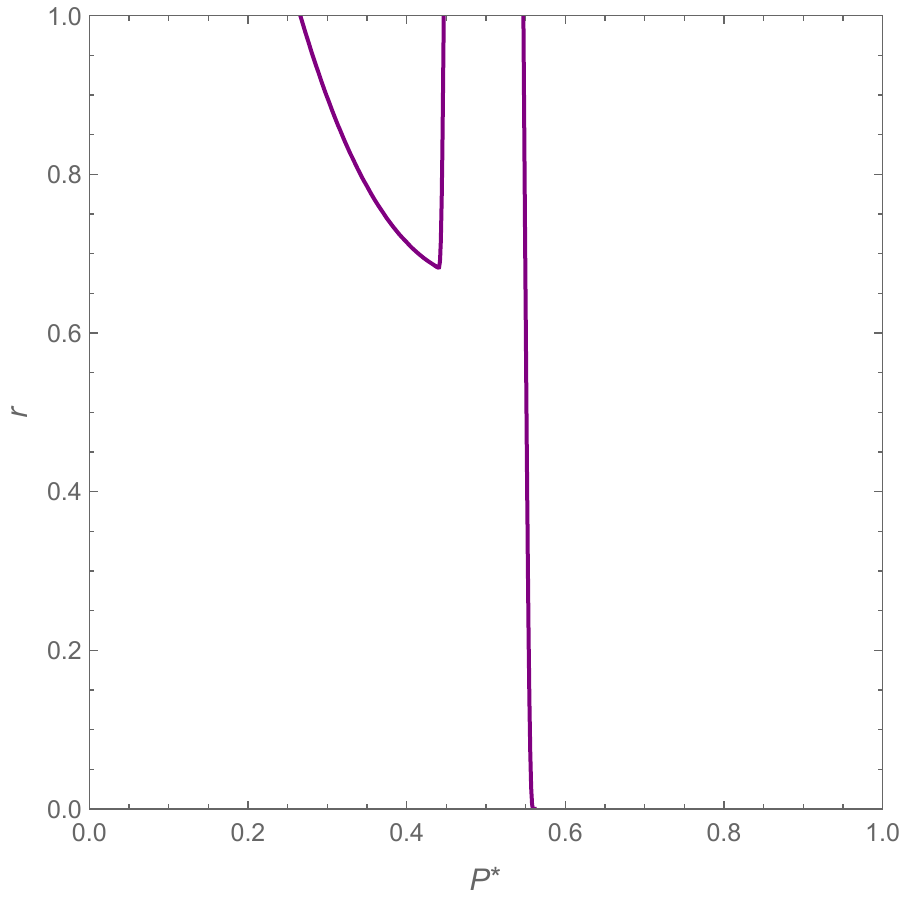}}
\caption{A schematic bifurcation diagram for varying $r$ in Example \ref{ex: second example}. Note that we did not include the vertical line on the $P^{\ast}$-axis in this figure.}
    \label{fig: example 2 bifurcations r}
\end{figure}
\begin{lemma}
    Varying $\epsilon$   results in saddle-node bifurcations of (\ref{eq: equilibria}). 
\end{lemma}

Figure \ref{fig: bifurcation diagrams} (b) presents a possible bifurcation diagram in $\epsilon$. 

\subsection{Varying $\epsilon$ and $r$}
In this section we discuss briefly what the three-dimensional bifurcation diagram (in $P$, $r$ and $\epsilon$) would look like.  This two-dimensional surface will give us the exact locations and the number of equilibria if we intersect it with a line $r =r_0$, $\epsilon = \epsilon_0$ for some fixed $r_0,\epsilon_0$. 

The picture in the general case is hard to construct, but it can be restored from its two-dimensional sections as drawn in Figure \ref{fig: bifurcation diagrams} (a) and (b). 
 \begin{remark}
     From Section \ref{sec: parameter r} it follows that connected components of this bifurcational graph will be manifolds: every section by the plane $\epsilon = const$ will be a cutoff of the graph of a smooth function depending on $P$, and the curve in question varies smoothly with varying $\epsilon$. 
 \end{remark}

For the general case, every line $P^{\ast}=constant$ will intersect the diagram in finite number of points (unless $r=0$ or $\epsilon$=0). The 'pointy ends" of the closed curves (the locus of values of $\epsilon$ and $r$ where bifurcations happen) will belong to a single curve. We can easily describe it. 

In order for the left and right hand sides of (\ref{eq: equilibria}) to be tangent to each other, two conditions must be fulfilled: both functions and their derivatives must have the same value ant some point $P$. This can be expressed as 
\begin{equation}
    \label{eq: whatis needed for bifurccation}
    \begin{cases}
        \epsilon \pi'(P)(1-P) + r = \pi(P),\\
        \epsilon\left(\pi''(P)(1-P) - \pi'(P)\right) = \pi'(P).
    \end{cases}
\end{equation}
This is a system of linear equations in $\epsilon$ and $r$ with a solution 
\begin{equation}
    \label{eq: bifurcation curve}
    \epsilon(P) =\frac{\pi'(P)}{\pi''(P)(1-P) - \pi'(P)},\\
    r(P) = \pi(P) - \frac{\left(\pi'(P)\right)^2(1-P)}{\pi''(P)(1-P) - \pi'(P)}.
\end{equation}
This curve contains all the values of $\epsilon$ and $r$ at which the number of equilibria changes for a given point $P^{\ast}$. Intersecting this curve changes the number of equilibria; moving along it ensures that the two graphs will be tangent at some point $P^{\ast}$. 
\section{Examples}
\label{sec: examples}
In the previous sections, we have analysed the stability and bifurcations with respect to the relevant parameters of the adaptive dynamics \eqref{eq: adaptive dynamics} for a general payoff function that satisfies the assumption stated in Section \ref{st: existence of zeros}. In this section, we provide two typical examples of such functions, one with singular equilibrium and the other with multiple equilibria. We also draw the bifurcation diagrams for the latter. 
 \begin{example}
\label{ex: concrete stuff}
\begin{figure}
    \centering
    \subfigure[$r=0.6$]{\includegraphics[scale=.35]{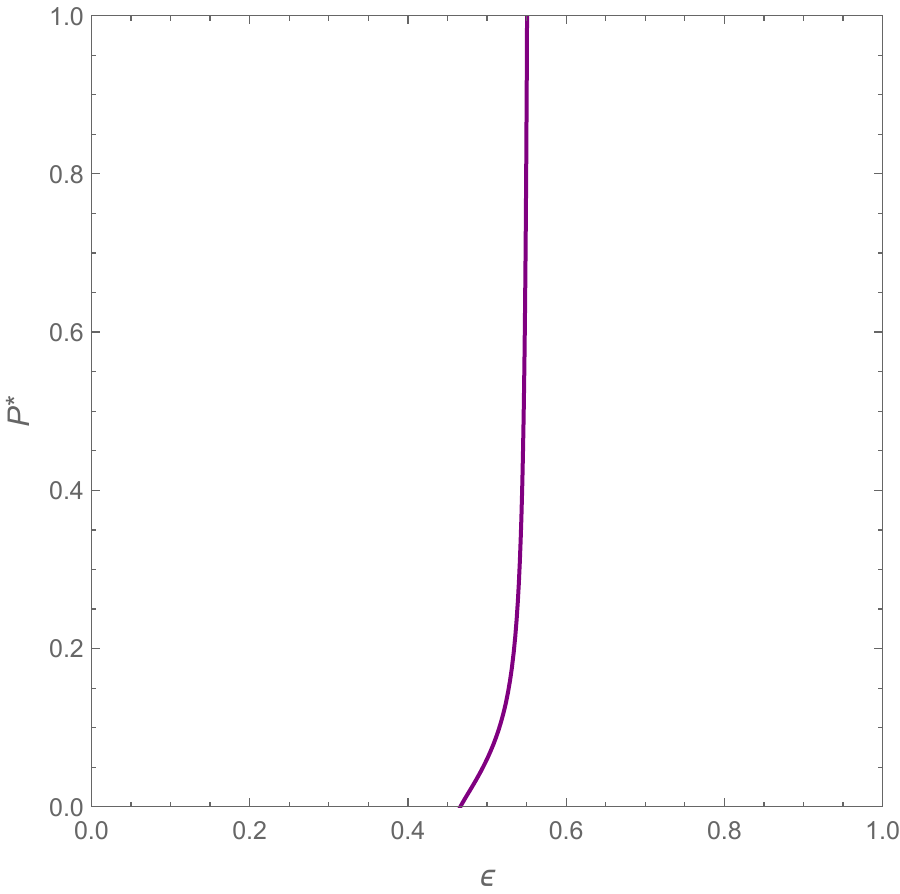}}
    \subfigure[$r=0.67$]{\includegraphics[scale=.35]{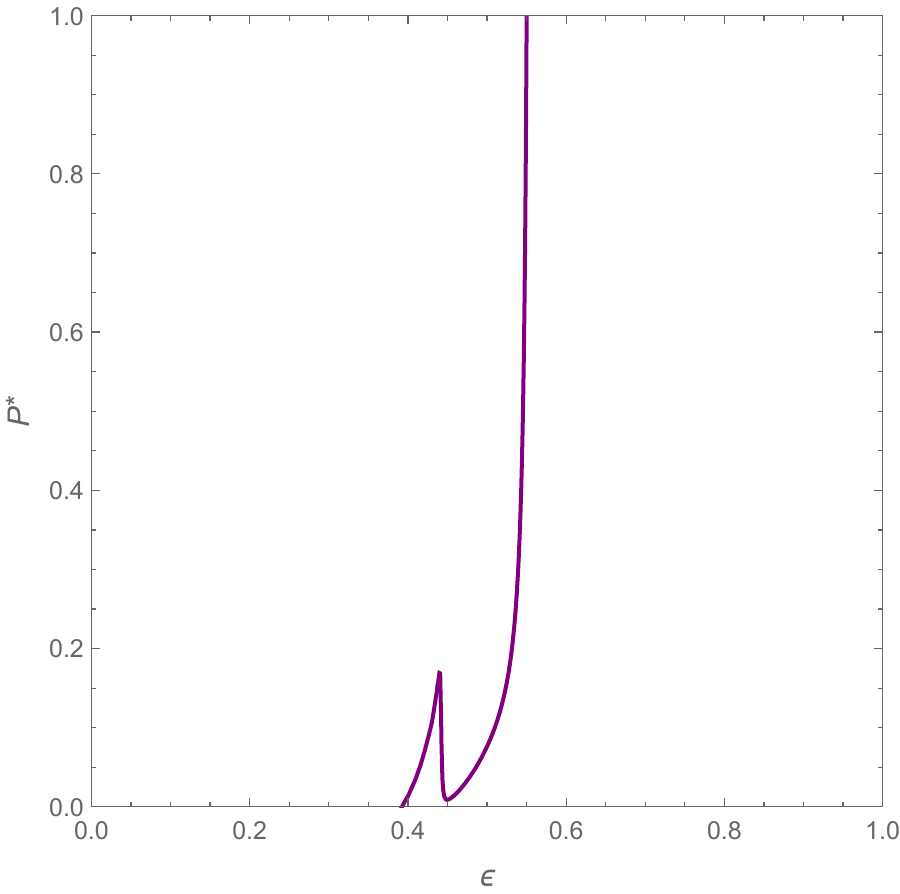}}
     \subfigure[$r=0.679$]{\includegraphics[scale=.35]{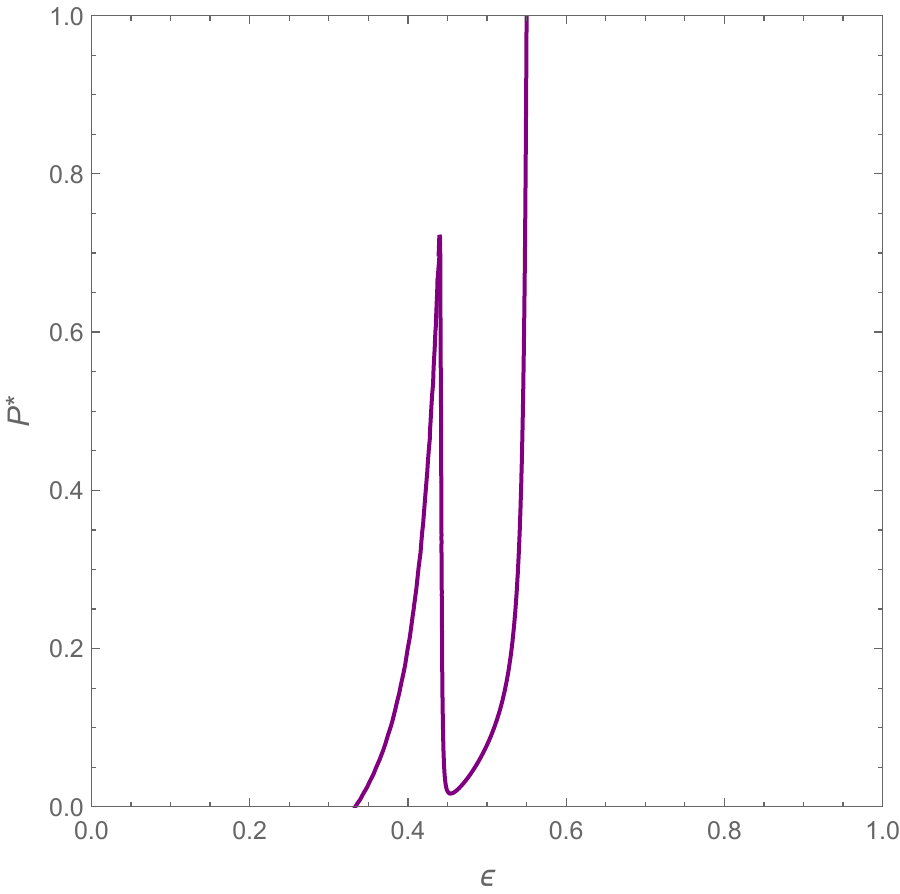}}
       \subfigure[$r=0.946$]{\includegraphics[scale=.35]{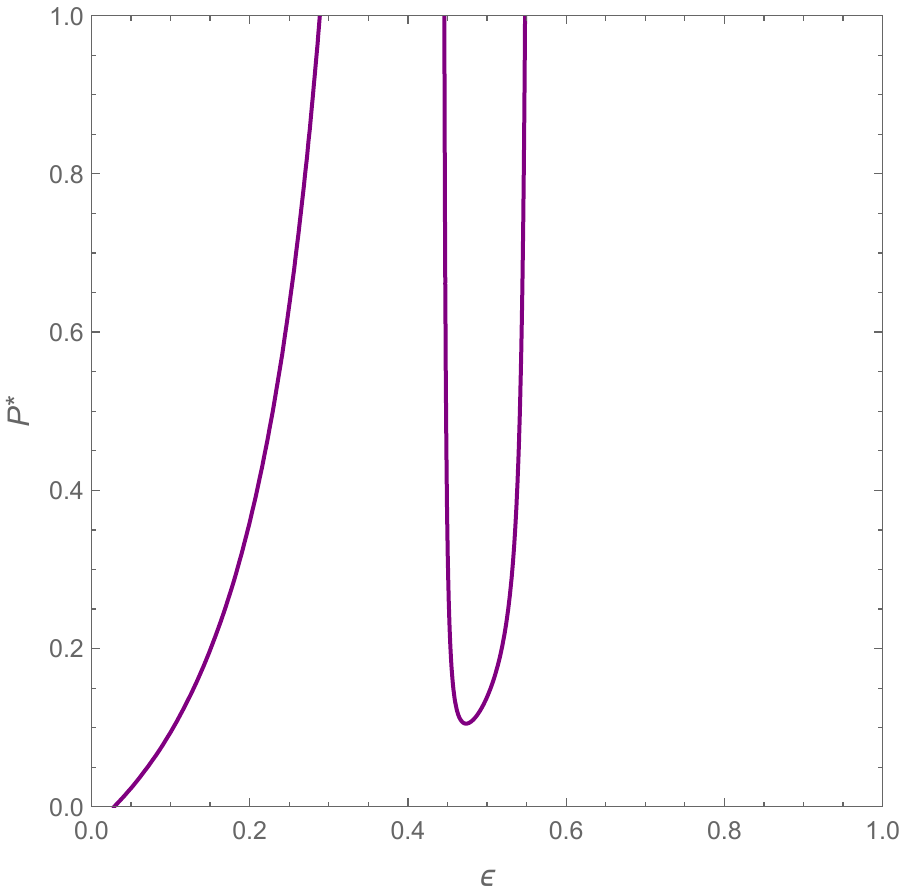}}
       \caption{Bifurcation diagrams for $\epsilon$ with varying values of $r$}
\label{fig: bifurcations in epsilon example 2}
\end{figure}
We want the function $\pi_p$ to be a smooth one; on the other hand, it ought to be equal to 0 when $p^{\ast}\le p\le 1$ for some $p^{\ast}$. The standard way to achieve this objective is the so-called "gluing" function. Below we describe its construction in the general case. 

Consider the function
\begin{equation}
    \label{eq: glue first}
    h(t) = \begin{cases}e^{\frac{1}{t^2-1}}, \ -1\le t \le 1\\
    0, \ \text{otherwise}
    \end{cases}
\end{equation}
This function in infinitely differentiable and smooth.

Let
\begin{equation}
    \label{eq: glue second}
    H(x) = \int
    \limits_{-1}^{x}h(t)\mathrm{d} t
\end{equation}
and 
\begin{equation}
    \label{eq: glue third}
    g(x) =\frac{H(2x-1)}{H(1)},
\end{equation}
also an infinitely differentiable and smooth function, that is 0 when $x\le 0$ and $1$ when $x\ge 1$. Additionally, $g(x)$ is monotonously increasing.

For the probability function, we use  the  example of $\pi(p)$ from \cite{crawford1991introduction}. Namely, we take
\begin{equation}
    \label{eq: pi p initial}
    \widetilde{\pi}(p) = 1 - \frac{1}{R_0(1-p)}
\end{equation}
with $R$ the basic reproductive ratio (considered to be a fixed number between 5 and 20). 

For concreteness, let $R=5$. Then the function $\pi(p) = 0$ when $p=\frac45$; we want our modified function to be zero from there.

To this end, we change $g(x)$ accordingly: firstly, it needs to become a decreasing function. Secondly, we want it to be zero when $x\ge \frac{4}{5}$; thirdly, we want the measure of the locus of the points where it decreases to be small - say, $\frac{1}{10}$, in order to minimise the effect that the gluing function has on the probability of becoming infected. 

To conform with these conditions, instead of $g(x)$ we must consider $\widetilde{g}(x):= 1 - g\left(10x-5\right)$. This function is 1 when $x\le \frac{7}{10}$, monotonously decreases when $\frac{7}{10}\le x\le \frac{4}{5}$ and assumes zero value when $x>\frac{4}{5}$.

Then the final version of the probability function is 
\begin{equation}
    \label{eq: pi final version}
    \pi(p) = \left(1 - \frac{1}{5(1-p)}\right)\widetilde{g}(p).
\end{equation}

This function is strictly  monotonously decreasing and smooth. It is 1 when $p=0$ and turns and stays 0 when $p\ge \frac45$. Its second derivative changes sign, meaning that the function $\pi'(P)$ is not monotonously increasing or decreasing. 

Setting $r = \frac15$ allows to  plot the two functions, and it can be shown that they only intersect at one point, see Figure \ref{fig: intersections example}. 

It can be ascertained numerically that the unique intersection point is $P\approx 0.7440219337$

\end{example}
\begin{example}
\label{ex: second example}
    We employ the same trick of inverting and constructing the glue function to suit our needs and multiply it by a decreasing function in order to get a suitable probability function.

    In this case, let $\widetilde{g}(x) = 1 - g(8x)$. The transition from 1 to 0  happens over the interval $\left[\frac{7}{16}, \frac{9}{16}\right]$. We multiply it by a decreasing function $\frac{1}{3} \left((1-2 x)^3+2\right)$ (the choice is motivated by its second derivative changing the sign) to get 
    \begin{equation}
    \label{eq: pi new}
        \pi(p) = \frac{1}{3} \widetilde{g}(p) \left((1-2 x)^3+2\right)
    \end{equation}
    Then its first derivative is 
    \begin{equation}
    \label{eq: pi derivative}
        \begin{split}
            \pi'(p) &= -2(1-2x)^2 \left(1-\frac{\int\limits_{-1}^{8(2x-1)}\exp \left(\frac{1}{t^2-1}\right)\mathrm{d}t}{H(1)}\right)\\&-\frac{16}{3} \exp \left(\frac{1}{(8 (2 x-1))^2-1}\right) \left((1-2 x)^3+2\right)\frac{1}{H(1)}
        \end{split}
    \end{equation}
Note that the number $\frac{1}{H(1)}$ is the coefficient from (\ref{eq: glue third}) and approximately equal to 2.2522836206907613.

The number of equilibria varies from 1 to 3 depending on the values of $r$ and $\epsilon$. Thus, the natural direction of investigation is to draw a bifurcation diagram-- that is Figures \ref{fig: example 2 bifurcations r} and \ref{fig: bifurcations in epsilon example 2}.

The former depicts the changes in the $r$-bifurcation diagram with varying values of $\epsilon$.  The  peak moves up with the increase in $\epsilon$ until it is no longer within the $[0,1]$ interval.

The second figure shows four $\epsilon$-bifurcation diagrams for various fixed values of $r$.  Note that with the increase of $r$ the beginning of the  curve becomes "pinched", and the new tip of the curve  moves upwards until it disappears and the diagram becomes disjoint.

Both Figures \ref{fig: bifurcations in epsilon example 2} and \ref{fig: example 2 bifurcations r} are sections of the 3-dimensional bifurcation diagram, as shown in Figure \ref{fig:3 dim bif ex 2}. It can be deduced from the bifurcation diagram that the number of solutions does not exceed 3. 

For example, when $\epsilon = 0.188$ and $r = 0.909$, the three solutions are, in increasing order, $P_0\approx0.7107347397$, $P_1\approx 0.4517078461$ and $P_2\approx0.5068541630 $

\begin{figure}
    \subfigure[View from the left]{\includegraphics[scale=.6]{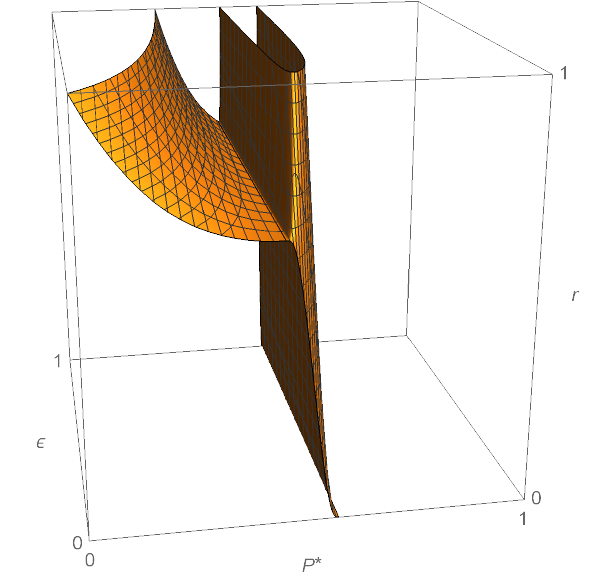}}
    \subfigure[View from the right]{\includegraphics[scale=.6
]{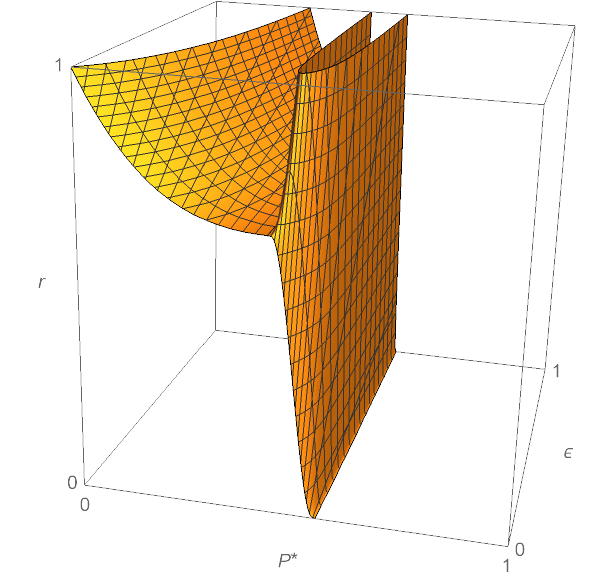}}
    \caption{Three dimensional bifurcation diagram for Example \ref{ex: second example}. The number of equilibria and their values can be obtained by intersecting this graph with some line $r = const, \epsilon = const$. }
    \label{fig:3 dim bif ex 2} 
\end{figure}
\end{example}
\section{Conclusion}
In this paper, we have examined the adaptive dynamics for the individual payoff game-theoretic model for vaccination \cite{bauch2004vaccination}. We have employed techniques from analysis and bifurcation theory. We have shown that at least one equilibrium must exist and that the bifurcation arising from the parameters can only lead to one type of bifurcations, which we described through the other parameter and the probability function $\pi(p)$.  We have concluded our investigation by examining closely two particular cases, focusing on the equilibrium for the first one and on bifurcation diagrams for the other.

We expect that the employment of the adaptive dynamics paves the way for further mathematically rigorous study of game-theoretic models for vaccination. In future works, we plan to extend this approach to mathematical models of infectious deceases such as Severe Acute Respiratory Syndrome (SARS) or COVID-19.  
\section*{Acknowledgement} The research of this paper was funded by the EPSRC (grant EP/Y008561/1) and the Royal International Exchange Grant IES-R3-223047.


\begin{thebibliography}{10}

\bibitem{WHO}
W{H}{O} website,.
\newblock \url{https://www.who.int/news-room/fact-sheets/detail/immunization-coverage}.
\newblock Accessed on 22.10.2024.

\bibitem{amanna2008protective}
I.~J. Amanna, I.~Messaoudi, and M.~K. Slifka.
\newblock Protective immunity following vaccination: how is it defined?
\newblock {\em Human vaccines}, 4(4):316--319, 2008.

\bibitem{bauch2004vaccination}
C.~T. Bauch and D.~J. Earn.
\newblock Vaccination and the theory of games.
\newblock {\em Proceedings of the National Academy of Sciences}, 101(36):13391--13394, 2004.

\bibitem{bauch2003group}
C.~T. Bauch, A.~P. Galvani, and D.~J. Earn.
\newblock Group interest versus self-interest in smallpox vaccination policy.
\newblock {\em Proceedings of the National Academy of Sciences}, 100(18):10564--10567, 2003.

\bibitem{brannstrom2013hitchhiker}
{\AA}.~Br{\"a}nnstr{\"o}m, J.~Johansson, and N.~Von~Festenberg.
\newblock The hitchhiker’s guide to adaptive dynamics.
\newblock {\em Games}, 4(3):304--328, 2013.

\bibitem{carpiano2023confronting}
R.~M. Carpiano, T.~Callaghan, R.~DiResta, N.~T. Brewer, C.~Clinton, A.~P. Galvani, R.~Lakshmanan, W.~E. Parmet, S.~B. Omer, A.~M. Buttenheim, et~al.
\newblock Confronting the evolution and expansion of anti-vaccine activism in the usa in the covid-19 era.
\newblock {\em The Lancet}, 401(10380):967--970, 2023.

\bibitem{chapman2012using}
G.~B. Chapman, M.~Li, J.~Vietri, Y.~Ibuka, D.~Thomas, H.~Yoon, and A.~P. Galvani.
\newblock Using game theory to examine incentives in influenza vaccination behavior.
\newblock {\em Psychological science}, 23(9):1008--1015, 2012.

\bibitem{crawford1991introduction}
J.~D. Crawford.
\newblock Introduction to bifurcation theory.
\newblock {\em Reviews of modern physics}, 63(4):991, 1991.

\bibitem{de2024dynamics}
H.~De~Silva and K.~Sigmund.
\newblock Dynamics of signaling games.
\newblock {\em SIAM Review}, 66(2):368--387, 2024.

\bibitem{hamborsky2015epidemiology}
J.~Hamborsky, A.~Kroger, et~al.
\newblock {\em Epidemiology and prevention of vaccine-preventable diseases, E-Book: The Pink Book}.
\newblock Public Health Foundation, 2015.

\bibitem{hausken2022game}
K.~Hausken and M.~Ncube.
\newblock A game theoretic analysis of competition between vaccine and drug companies during disease contraction and recovery.
\newblock {\em Medical Decision Making}, 42(5):571--586, 2022.

\bibitem{hofbauer1998evolutionary}
J.~Hofbauer and K.~Sigmund.
\newblock {\em Evolutionary games and population dynamics}.
\newblock Cambridge university press, 1998.

\bibitem{kim2016advisory}
D.~K. Kim, C.~B. Bridges, K.~H. Harriman, and A.~C. on~Immunization~Practices†.
\newblock Advisory committee on immunization practices recommended immunization schedule for adults aged 19 years or older: United states, 2016.
\newblock {\em Annals of internal medicine}, 164(3):184--194, 2016.

\bibitem{kuznetsov1998elements}
Y.~A. Kuznetsov, I.~A. Kuznetsov, and Y.~Kuznetsov.
\newblock {\em Elements of applied bifurcation theory}, volume 112.
\newblock Springer, 1998.

\bibitem{mckee2016exploring}
C.~McKee and K.~Bohannon.
\newblock Exploring the reasons behind parental refusal of vaccines.
\newblock {\em The journal of pediatric pharmacology and therapeutics}, 21(2):104--109, 2016.

\bibitem{traulsen2023individual}
A.~Traulsen, S.~A. Levin, and C.~M. Saad-Roy.
\newblock Individual costs and societal benefits of interventions during the covid-19 pandemic.
\newblock {\em Proceedings of the National Academy of Sciences}, 120(24):e2303546120, 2023.

\bibitem{van2002adaptive}
M.~Van~Baalen, U.~Dieckmann, J.~Metz, M.~Sabelis, and K.~Sigmund.
\newblock Adaptive dynamics of infectious diseases: in pursuit of virulence management.
\newblock {\em Contact networks and the evolution of virulence}, pages 85--103, 2002.

\bibitem{10.1001/archpedi.159.12.1136}
F.~Zhou, J.~Santoli, M.~L. Messonnier, H.~R. Yusuf, A.~Shefer, S.~Y. Chu, L.~Rodewald, and R.~Harpaz.
\newblock {Economic Evaluation of the 7-Vaccine Routine Childhood Immunization Schedule in the United States, 2001}.
\newblock {\em Archives of Pediatrics \& Adolescent Medicine}, 159(12):1136--1144, 12 2005.

\end{thebibliography}
\end{document}